\documentclass[leqno]{siamltex704}
\usepackage{amsmath,amssymb}
\usepackage{epsfig}
\usepackage{graphicx}

\def\pK{{\partial K}}

\def\bn{{\bf n}}
\def\vn{{\bf n}}

\def\3bar{{|\hspace{-.02in}|\hspace{-.02in}|}}

\title{Interior penalty discontinuous Galerkin method on very general polygonal and polyhedral meshes}
\author{Mu Lin\thanks{Department of
Mathematics, University of Arkansas at Little Rock, Little Rock, AR
72204 (lxmu@ualr.edu).}\and Junping Wang\thanks{Division of Mathematical Sciences, National
Science Foundation, Arlington, VA 22230 (jwang@\break nsf.gov). The
research of Wang was supported by the NSF IR/D program, while
working at the Foundation. However, any opinion, finding, and
conclusions or recommendations expressed in this material are those
of the author and do not necessarily reflect the views of the
National Science Foundation.}
\and Yanqiu Wang\thanks{Department of Mathematics,
Oklahoma State University, Stillwater, OK 74075 (yqwang@math.okstate.edu).}
\and Xiu Ye\thanks{Department of
Mathematics, University of Arkansas at Little Rock, Little Rock, AR
72204 (xxye@ualr.edu). This research was supported in part by
National Science Foundation Grant DMS-1115097.}}

\begin{document}
\maketitle

\begin{abstract}
This paper focuses on interior penalty discontinuous Galerkin
methods for second order elliptic equations on very general
polygonal or polyhedral meshes. The mesh can be composed of any
polygons or polyhedra which satisfies certain shape regularity
conditions characterized in a recent paper by two of the authors in
\cite{WangYe2012}. Such general meshes have important application in
computational sciences. The usual $H^1$ conforming finite element
methods on such meshes are either very complicated or impossible to
implement in practical computation. However, the interior penalty
discontinuous Galerkin method provides a simple and effective
alternative approach which is efficient and robust. This article
provides a mathematical foundation for the use of interior penalty
discontinuous Galerkin methods in general meshes.
\end{abstract}

\begin{keywords}
discontinuous Galerkin, finite element, interior penalty,
second-order elliptic equations, hybrid mesh.
\end{keywords}

\begin{AMS}
65N15, 65N30.
\end{AMS}

\section{Introduction}
Most finite element methods are constructed on triangular and quadrilateral meshes,
or on tetrahedral, hexahedral, prismatic, and pyramidal meshes.
To extend the idea of the finite element method into meshes employing general polygonal and polyhedral elements,
one immediately faces the problem of choosing suitable discrete spaces on general polygons and polyhedrons.
This issue has rarely been addressed in the past, partly because it can usually be circumvented by dividing
the polygon or polyhedron into sub-elements using only one or two basic shapes.
However, allowing the use of general polygonal and polyhedral elements does provide
more flexibility, especially for complex geometries or problems with certain physical constraints.
One of such example is the modeling of composite microstructures in material sciences.
A well-known solution to this problem is the Voronoi cell finite element method
\cite{Ghosh94, Ghosh95, Ghosh04, Moorthy98},
in which the mesh is composed of polygons or polyhedrons representing
the grained microstructure of the given material.
The main difficulty of constructing conforming finite element methods on Voronoi meshes is that,
the finite element space has to be carefully chosen so that it is continuous along interfaces.
Although the constructions on triangles, quadrilaterals, or three-dimensional simplexes are straight forward,
it is not easy for general polygons and polyhedrons.
Probably the only practically used solution is the rational polynomial interpolants proposed by
Wachspress \cite{Wachspress75}, in which
rational basis functions are defined using distances from several ``nodes''.
An important constraint in the construction of the Wachspress basis is that, the rational basis functions need to be
piecewise linear along the boundary of every element, in order to ensure $H^1$ conformity of the finite element space.
This not only limits the approximation order of the entire Wachspress finite element space, but also
complicates the construction.
The Wachspress element has gained a renewed interest recently \cite{Dasgupta03, Dasgupta03b, Sukumar06}.
However, as we have pointed out above, its construction is complicated and usually requires the aid of computational
algebraic systems such as Maple.

Another practically important issue is to define finite element methods on hybrid meshes.
Hybrid meshes are frequently used nowadays. It can handle complicated geometries,
and can sometimes reduce the total number of unknowns.
Another possible reason for using the hybrid mesh is that, some engineers argue that in three-dimensions,
a hexahedral mesh yields more accurate solution than a tetrahedral mesh for the same geometry \cite{Yamakawa03, Yamakawa09},
as partly verified by numerical experiments.
However, pure hexahedral meshes lack the ability of handling complicated geometries.
Hence a hybrid mesh becomes a welcomed compromise between accuracy and flexibility.
For conforming finite element methods based on hybrid meshes, continuity requirements on
interfaces must be satisfied. Such a coupling is straight-forward for the $H^1$-conforming finite elements
on a triangular-quadrilateral hybrid mesh.
However, for three-dimensional meshes, high order finite elements, or other complicated finite element spaces,
it usually requires special treatments.

An alternative solution, that can address both issues mentioned above,
is to use the weak Galerkin method proposed in \cite{WangYe2012}.
The weak Galerkin method uses discontinuous piecewise polynomials inside each element and on the interfaces
to approximate the variational solution. In \cite{WangYe2012}, the authors have proved optimal convergence
of the weak Galerkin method for the mixed formulation of second order elliptic equations on very general
polygonal and polyhedral meshes. Most of the existing error analysis of finite element methods assume triangular, quadrilateral, or some commonly-seen three-dimensional meshes. To our knowledge, it is the first time that optimal convergence for the finite element solutions has been rigorously proved in \cite{WangYe2012} for general meshes of arbitrary polygons and polyhedrons.

The discontinuous Galerkin method imposes the interface continuity weakly,
and is known to be able to handle non-conformal, hybrid meshes as well as a variety of basis functions.
There have been many research works in this direction, for example, nodal discontinuous Galerkin methods
\cite{Bergot10, Cohen00, Hesthaven00} for hyperbolic conservation laws.  However, we would like to point out that so far there has been no theoretical analysis on the convergence rate of discontinuous Galerkin method, on very general polygonal or polyhedral meshes yet.
Motivated by the work in \cite{WangYe2012}, here we would like to fill the gap. The objective of this paper is to establish
the theoretical analysis of the interior penalty discontinuous Galerkin method \cite{Arnold02}
for elliptic equations on very general meshes and discrete spaces.

The paper is organized as follows. In Section 2, we briefly describe the interior penalty discontinuous Galerkin method
in an abstract setting. In Section 3, several assumptions on the discrete spaces are listed, which form
a minimum requirement for the well-posedness and the approximation property of the discrete formulation.
Abstract error estimations are given. In Section 4, we discuss choices of meshes and discrete spaces that satisfy
the assumptions given in Section 3. Finally, numerical results are presented in Section 5.

\section{The model problem and the interior penalty method}
Consider the model problem
\begin{equation} \label{eq:ellipticeq}
\begin{cases}
-\Delta u=f\qquad &\mbox{in }\Omega,\\
u=0 &\mbox{on }\partial\Omega,
\end{cases}
\end{equation}
where $\Omega\in\mathbb{R}^d(d=2,3)$ is a closed domain with Lipschitz continuous boundary,
and $f\in L^2(\Omega)$.

For any subdomain $K\subset \Omega$ with Lipschitz continuous boundary, we use the standard
definition of Sobolev spaces $H^s(K)$ with $s\ge 0$ (e.g., see
\cite{adams, ciarlet} for details). The associated inner product,
norm, and seminorms in $H^s(K)$ are denoted by
$(\cdot,\cdot)_{s,K}$, $\|\cdot\|_{s,K}$, and $|\cdot|_{s,K}$, respectively.
When $s=0$, $H^0(K)$ coincides with the space
of square integrable functions $L^2(K)$. In this case, the subscript
$s$ is suppressed from the notation of norm, semi-norm, and inner
products. Furthermore, the subscript $K$ is also suppressed when
$K=\Omega$. Finally, all above notations can easily be extended
to any $e\subset \partial K$.
For the $L^2$ inner product on $e$, we usually denote it as
$\langle\cdot,\cdot\rangle_{e}$ in stead of $(\cdot,\cdot)_{e}$, as it can be
replaced by the duality pair when needed.

For simplicity, we assume that $\Omega$ satisfy certain conditions such that
Equation (\ref{eq:ellipticeq}) has at least $H^{r}$ regularity with $r>3/2$,
that is, the solution to Equation (\ref{eq:ellipticeq}) satisfies $u\in H^{r}(\Omega)$ and
\begin{equation} \label{eq:regularity}
  \|u\|_r \le C_R \|f\|.
\end{equation}
This assumption is standard in the practice of interior penalty discontinuous Galerkin methods,
as it ensures that the exact solution $u$ also satisfies the discontinuous Galerkin formulation,
and thus the a priori error estimation can be easily derived in a Lax-Milgram framework.
However, such a regularity assumption is not necessary in the practice of interior penalty methods.
A well-known technique, which was first proposed by
Gudi \cite{Gudi10}, is to use a posteriori error estimation to derive an a priori error estimation
for the interior penalty method, with only minimum regularity requirement $u\in H^1(\Omega)$.
We believe that the same technique applies for the general polygonal and polyhedral meshes,
as long as a working a posteriori error estimation is available.
However, here we choose to completely skip this issue, as it is not the main purpose of this paper.

Assume that for all set $K$ discussed in this paper, including $\Omega$ itself,
the unit outward normal vector $\bn$ is defined almost everywhere on $\partial K$.
Note this is true for all polygonal and polyhedral elements with Lipschitz continuous boundaries.
Since the exact solution $u\in H^{r}(\Omega)$ with $r>3/2$, it is clear that for any smooth function $v$ defined on $K$,
$$
(\nabla u, \, \nabla v)_K - \langle\nabla u\cdot\vn,\, v\rangle_{\partial K} = (f,\, v)_K,
$$
where $(\cdot,\cdot)_K$ is the $L^2$-inner product in
$L^2(K)$ and  $\langle\cdot,\cdot\rangle_\pK$ is the $L^2$-inner product in
$L^2(\pK)$

Let $\mathcal{T}_h$ be a partition of the domain $\Omega$ into non-overlapping subdomains/elements, each with Lipschitz continuous boundary.
Here $h$ denotes the characteristic size of the partition, which will be defined in details later.
The interior interfaces are denoted by $e = \bar{K_1}\cap\bar{K_2}$, where $K_1$, $K_2\in \mathcal{T}_h$.
Boundary segments are similarly denoted by $e = \bar{K}\cap\partial\Omega$, where $K\in \mathcal{T}_h$.
Denote by $\mathcal{E}_h$ the set of all interior interfaces and boundary segments in $\mathcal{T}_h$,
and by $\mathcal{E}_h^0=\mathcal{E}_h\setminus\partial\Omega$ the set of all interior interfaces.
For every $K\in\mathcal{T}_h$, let $|K|$ be the area/volume of $K$,
and for every $e\in \mathcal{E}_h$, let $|e|$ be its length/area.
Denote $h_e$  the diameter of $e\in \mathcal{E}_h$ and $h_K$ the diameter of $K\in \mathcal{T}_h$.
Clearly, when $e\subset \partial K$, we have $h_e\le h_K$.
Finally, define $h=\max_{K\in \mathcal{T}_h} h_K$ to be the characteristic mesh size.

Notice that $\mathcal{T}_h$ defined above is a very general mesh/partition on $\Omega$, as we do not
specify the shape and conformal property of $K\in \mathcal{T}_h$.
The interior penalty discontinuous Galerkin (IPDG) method can be extended to
such a general mesh, without any modification of the formulation.
However, to ensure its approximation rate, certain conditions must be imposed
on $\mathcal{T}_h$ and the discrete function spaces.
In this paper, we are interested in discussing the minimum requirements of such conditions.
First, we shall give the formulation of the interior penalty discontinuous Galerkin method.

Let $V_K$ be a finite dimensional space of smooth functions defined on $K \in \mathcal{T}_h$. Define
$$
V_h=\{v\in L^2(\Omega):v|_K\in V_K,\textrm{ for all } K\in\mathcal{T}_h\},
$$
and
$$
V(h)=V_h+\left( H_0^1(\Omega)\cap \prod_{K\in\mathcal{T}_h} H^{r}(K) \right),\qquad \textrm{where }r>\frac{3}{2}.
$$
For any internal interface $e = \bar{K_1}\cap\bar{K_2} \in \mathcal{E}_h$,
let $\bn_1$ and $\bn_2$ be the unit outward normal vectors on $e$, associated with ${K_1}$ and ${K_2}$, respectively.
For $v\in V(h)$, define the average $\{\nabla v\}$ and jump $[v]$ on $e$ by
$$
\{\nabla v\} = \frac{1}{2}\left(\nabla v|_{K_1} + \nabla v|_{K_2} \right),\qquad
[v] = v|_{K_1} \bn_1 + v|_{K_2} \bn_2.
$$
On any boundary segment $e= \bar{K}\cap\partial\Omega$, the above definitions of average and jump need to be modified:
$$
\{\nabla v\} = \nabla v|_{K},\qquad  [v] = v|_{K} \bn_K,
$$
where $\bn_K$ is the unit outward normal vector on $e$ with respect to $K$.

Define a bilinear form on $V(h)\times V(h)$ by
$$
\begin{aligned}
  A(u,v) =& \sum_{K\in \mathcal{T}_h}(\nabla u,\nabla v)_K-\sum_{e\in\mathcal{E}_h}\langle\{\nabla u\},\, [v]\rangle_e \\
  &\quad -\delta \sum_{e\in\mathcal{E}_h}\langle\{\nabla v\},\, [u]\rangle_e + \alpha \sum_{e\in\mathcal{E}_h}\frac{1}{h_e} \langle[u],\, [v]\rangle_e,
\end{aligned}
$$
where $\delta = \pm 1,\, 0$ and $\alpha>0$.
when $\delta=1$, the bilinear form $A(\cdot,\cdot)$ is symmetric.
The constant $\alpha$ is usually required to be large enough, but still independent of the mesh size $h$,
in order to guarantee the well-posedness of the discontinuous Galerkin formulation. Details will be given later.

It is clear that the exact solution $u$ to Equation (\ref{eq:ellipticeq}) satisfies
\begin{equation} \label{eq:dg-exactsol}
A(u,v) = (f,v)\qquad\textrm{for all } v\in V_h,
\end{equation}
as $[u]$ vanishes on all $e\in\mathcal{E}_h$.
Hence the following interior penalty discontinuous Galerkin formulation is consistent with Equation (\ref{eq:ellipticeq}):
find $u_h \in V_h$ satisfying
\begin{equation} \label{eq:dg}
A(u_h,v) = (f,v)\qquad\textrm{for all } v\in V_h.
\end{equation}

Finally, we would like to point out that the formulation (\ref{eq:dg}) is computable,
as long as each finite dimensional space $V_K$ has a clearly defined and computable basis.
\section{Abstract theory}
Define a norm $\3bar\cdot\3bar$ on $V(h)$ as following:
\begin{eqnarray}
\3bar v\3bar^2=\sum_{K\in\mathcal{T}_h}\|\nabla v\|_K^2+\sum_{e\in\mathcal{E}_h}h_e\|\{\nabla v\}\|_e^2+\alpha\sum_{e\in\mathcal{E}_h}\frac{1}{h_e}\|[v]\|_e^2.
\end{eqnarray}
By the Poincar\'{e} inequality, $\3bar\cdot\3bar$ is obviously a well-posed norm on $V(h)$.

Next, we give a set of assumptions, which form the minimum requirements guaranteeing
the well-posedness and
the approximation properties of the interior penalty discontinuous Galerkin method.
\begin{itemize}
\item [{\bf I1}] (The trace inequality) There exists a positive constant $C_T$ such that for all $K\in\mathcal{T}_h$ and $\theta\in H^1(K)$, we have
  \begin{equation} \label{eq:TraceIn}
    \|\theta\|_{\partial K}^2\le C_{T}(h_K^{-1}\|\theta\|_K^2+h_K\|\nabla\theta\|_K^2).
  \end{equation}
\item [{\bf I2}] (The inverse inequality) There exists a positive constant $C_I$ such that for all $K\in\mathcal{T}_h$,
  $\phi\in V_K$ and $\phi\in \frac{\partial}{\partial x_i}V_K$ where $i=1,\ldots, d$, we have
  \begin{equation} \label{eq:InverseIn}
    \|\nabla\phi\|_K\le C_I\, h_K^{-1}\|\phi\|_K.
  \end{equation}
\item [{\bf I3}] (The approximability) There exist positive constants $s$ and $C_A$ such that
  for all $v\in H^{s+1}(\Omega)$, we have
  \begin{equation} \label{eq:approximability}
    \inf_{\chi_h\in V_h} \3bar v-\chi_h \3bar \le C_A \left(\sum_{K\in \mathcal{T}_h} h_K^{2s} \|v\|_{s+1,K}^2\right)^{1/2}.
  \end{equation}
\end{itemize}

The abstract theory of the interior penalty discontinuous Galerkin method can be entirely based on Assumptions {\bf{I1}}-{\bf{I3}}.

\begin{lemma} \label{lem:wellposedness}
Assume {\bf{I1}}-{\bf{I2}} hold. The bilinear form $A(\cdot,\cdot)$ is bounded in $V(h)$, with respect to the norm $\3bar\cdot\3bar$.
Indeed,
$$
A(u,v)\le \frac{1+\alpha}{\alpha} \3bar u\3bar\, \3bar v\3bar\qquad \textrm{for all } u,\, v\in V(h).
$$
Furthermore, denote $C_1 = C_T(1+C_I)^2$.
Then for any constant $0<C<1$ and $\alpha \ge \frac{(1+\delta)^2C_1}{4(1-C)^2}$, the bilinear form $A(\cdot,\cdot)$ is coercive on $V_h$.
That is,
$$
  A(v, v) \ge \frac{C}{1+C_1} \3bar v\3bar^2\qquad  \textrm{for all } v\in V_h.
$$
\end{lemma}

\begin{proof}
  The boundedness of $A(\cdot,\cdot)$ follows immediately from the Schwarz inequality.
Here we only prove the coercivity.
First, notice that for all $v\in V_h$, by assumptions {\bf{I1}}-{\bf{I2}} and the fact that
$h_e\le h_K$ for all $e\in \partial K\cap \mathcal{E}_h$,
$$
\begin{aligned}
\sum_{e\in\mathcal{E}_h}h_e\|\{\nabla v\}\|_e^2 & \le \sum_{K\in\mathcal{T}_h} \left(\sum_{e\in \partial K\cap \mathcal{E}_h} h_e \|\nabla v\|_e^2\right) \\
&\le \sum_{K\in\mathcal{T}_h} h_K \|\nabla v\|_{\partial K}^2  \\
& \le  \sum_{K\in\mathcal{T}_h} h_K  \bigg( C_T(1+C_I^2) h_K^{-1} \|\nabla v\|_K^2 \bigg) \\
& = C_1 \sum_{K\in\mathcal{T}_h} \|\nabla v\|_K^2.
\end{aligned}
$$
Then, by the Schwarz inequality, the Young's inequality and assumptions {\bf{I1}}-{\bf{I2}}, we have
$$
\begin{aligned}
\sum_{e\in\mathcal{E}_h}\langle\{\nabla v\},\, [v]\rangle_e &\le
   \varepsilon \sum_{e\in\mathcal{E}_h}h_e\|\{\nabla v\}\|_e^2 + \frac{1}{4\varepsilon} \sum_{e\in\mathcal{E}_h}\frac{1}{h_e} \|[v]\|_e^2 \\
  &\le  \varepsilon C_1 \sum_{K\in\mathcal{T}_h} \|\nabla v\|_K^2  + \frac{1}{4\varepsilon\alpha} \left(\alpha \sum_{e\in\mathcal{E}_h}\frac{1}{h_e} \|[v]\|_e^2 \right),
\end{aligned}
$$
where $\varepsilon$ is chosen to be $\frac{1-C}{(1+\delta)C_1}$ for any given constant $0<C<1$.
Clearly, for such an $\varepsilon$, we have $1-(1+\delta)\varepsilon C_1 = C$ and
$$
  1-\frac{1+\delta}{4\varepsilon\alpha} \ge C \quad \Longleftrightarrow \quad \alpha \ge \frac{(1+\delta)^2C_1}{4(1-C)^2}.
$$
Combine the above and let $\alpha \ge \frac{(1+\delta)^2C_1}{4(1-C)^2}$, we have
$$
\begin{aligned}
  A(v, v) &= \sum_{K\in\mathcal{T}_h} \|\nabla v\|_K^2 - (1+\delta) \sum_{e\in\mathcal{E}_h}\langle\{\nabla v\},\, [v]\rangle_e
    + \alpha \sum_{e\in\mathcal{E}_h}\frac{1}{h_e} \|[v]\|_e^2 \\
    &\ge \bigg( 1-(1+\delta)\varepsilon C_1\bigg)\sum_{K\in\mathcal{T}_h} \|\nabla v\|_K^2
    + \bigg(1-\frac{1+\delta}{4\varepsilon\alpha}\bigg)\alpha \sum_{e\in\mathcal{E}_h}\frac{1}{h_e} \|[v]\|_e^2 \\
    &\ge C \left( \sum_{K\in\mathcal{T}_h} \|\nabla v\|_K^2 +  \alpha \sum_{e\in\mathcal{E}_h}\frac{1}{h_e} \|[v]\|_e^2 \right)\\
    &\ge \frac{C}{1+C_1} \3bar v\3bar^2.
\end{aligned}
$$
\end{proof}

Lemma \ref{lem:wellposedness} guarantees the existence and uniqueness of the solution to
Equation (\ref{eq:dg}). In the rest of this paper, we shall always assume $\alpha \ge \frac{(1+\delta)^2C_1}{4(1-C)^2}$.
Let $u$ and $u_h$ be the solution to equations (\ref{eq:ellipticeq}) and (\ref{eq:dg}), respectively.
By subtracting (\ref{eq:dg-exactsol}) from (\ref{eq:ellipticeq}), one gets the standard orthogonality property of the error,
$$
A(u-u_h,\, v_h) = 0\qquad\textrm{for all } v\in V_h.
$$
Then clearly, for all $\chi_h\in V_h$,
$$
\begin{aligned}
\3bar \chi_h - u_h\3bar^2 &\le \frac{1+C_1}{C}A(\chi_h - u_h,\, \chi_h - u_h) \\
&= \frac{1+C_1}{C}A(\chi_h - u,\, \chi_h - u_h) \\
&\le \frac{(1+C_1)(1+\alpha)}{C\alpha} \3bar \chi_h - u\3bar \, \3bar \chi_h - u_h\3bar.
\end{aligned}
$$
Then, using the triangle inequality,
$$
\begin{aligned}
\3bar u-u_h\3bar &\le \inf_{\chi_h\in V_h} \bigg( \3bar u-\chi_h\3bar + \3bar \chi_h-u_h\3bar \bigg) \\
&\le \bigg(1 + \frac{(1+C_1)(1+\alpha)}{C\alpha}  \bigg) \inf_{\chi_h\in V_h}\3bar u-\chi_h\3bar.
\end{aligned}
$$
Combine this with assumption {\bf{I3}}, we get the following abstract error estimation:
\begin{theorem}
  Assume {\bf{I1}}-{\bf{I3}} hold, $C$ be a given constant in $(0,1)$ and $\alpha \ge \frac{(1+\delta)^2C_1}{4(1-C)^2}$.
Let $u$ and $u_h$ be the solution to equations (\ref{eq:ellipticeq}) and (\ref{eq:dg}), respectively.
Then
$$
\3bar u-u_h\3bar \lesssim C_A\bigg(1 +  \frac{(1+C_1)(1+\alpha)}{C\alpha} \bigg)
  \left(\sum_{K\in \mathcal{T}_h} h_K^{2s} \|u\|_{s+1,K}^2\right)^{1/2}.
$$
\end{theorem}

Finally, we derive the $L^2$ error estimation by using the standard duality argument.
Let $\delta=1$, that is, the bilinear form $A(\cdot,\cdot)$ is symmetric. Consider the following problem
$$
\begin{cases}
-\Delta \phi=u-u_h\qquad &\mbox{in }\Omega,\\
\phi=0 &\mbox{on }\partial\Omega.
\end{cases}
$$
Here again, we assume that the domain $\Omega$ satisfies certain condition such that $\phi$ has $H^r$ regularity, with $r>3/2$.
Let $\phi_h\in V_h$ be an approximation to $\phi$ such that they satisfy Assumption {\bf{I3}}.
Clearly
$$
\begin{aligned}
  \|u-u_h\|^2 &= (-\Delta \phi,\, u-u_h) = \sum_{K\in \mathcal{T}_h}(\nabla \phi,\nabla (u-u_h))_K-\sum_{e\in\mathcal{E}_h}\langle\{\nabla \phi\},\, [u-u_h]\rangle_e \\
  &= A(\phi, \, u-u_h) = A(\phi-\phi_h,\, u-u_h)\\
  &\le \frac{1+\alpha}{\alpha} \3bar \phi-\phi_h \3bar\, \3bar u-u_h\3bar \\
  &\le \frac{1+\alpha}{\alpha}C_A  \left(\sum_{K\in \mathcal{T}_h} h_K^{2\min\{r-1,s\}} \|\phi\|_{\min\{r,s+1\},K}^2\right)^{1/2} \3bar u-u_h\3bar.
\end{aligned}
$$
This gives the following theorem
\begin{theorem}
  Assume {\bf{I1}}-{\bf{I3}} hold, $\delta=1$, $C$ be a given constant in $(0,1)$, $\alpha \ge \frac{(1+\delta)^2C_1}{4(1-C)^2}$, and
the elliptic equation (\ref{eq:ellipticeq}) has $H^r$ regularity with $r>3/2$.
Let $u$ and $u_h$ be the solution to equations (\ref{eq:ellipticeq}) and (\ref{eq:dg}), respectively.
Then
$$
\|u-u_h\| \le \frac{1+\alpha}{\alpha}C_AC_R h^{\min\{r-1,s\}} \3bar u-u_h\3bar.
$$
\end{theorem}

\section{Requirements on meshes and discrete spaces}
On triangular or quadrilateral meshes, the usual tool for proving assumptions {\bf{I1}}-{\bf{I3}} is to use a scaling
argument built on affine transformations. However, on general polygons and polyhedrons, it is not clear how to define such affine transformations. The assumptions  {\bf{I1}}-{\bf{I3}} were first validated in \cite{WangYe2012} for general polygonal and polyhedral meshes that satisfy a set of conditions introduced in \cite{WangYe2012}. Such conditions can be stated as follows.

All the elements of $\mathcal{T}_h$ are assumed to be closed and simply connected polygons or polyhedrons.
We make the following shape regularity assumptions for the partition $\mathcal{T}_h$.

\begin{itemize}
\item[{\bf A1:}] Assume that there exist two positive constants $\rho_v$ and $\rho_e$ such that for every element $K\in\mathcal{T}_h$ and $e\in \mathcal{E}_h$,
we have
\begin{eqnarray}
\rho_vh_K^d\le |K|,\ \ \rho_eh_e^{d-1}\le |e|.
\end{eqnarray}
\item[{\bf A2:}] Assume that there exists a positive constant $\kappa$ such that for every element $K\in\mathcal{T}_h$ and $e\in \partial K\cap \mathcal{E}_h$,
 we have
\begin{eqnarray}
\kappa h_K\le h_e.
\end{eqnarray}
\item[{\bf A3:}] Assume that for every $K\in\mathcal{T}_h,$ and $e\in\partial K \cap \mathcal{E}_h$,
there exists a pyramid $P(e,K,A_e)$ contained in $K$ such that its base is identical with $e$, its apex is $A_e\in K$,
and its height is proportional to $h_K$ with a proportionality constant $\sigma_e$ bounded away from a fixed positive number $\sigma^*$ from below.
In other words, the height of the pyramid is given by $\sigma_eh_K$ such that $\sigma_e\ge \sigma^*>0.$
The pyramid is also assumed to stand up above the base $e$ in the sense that the angle between the vector ${\bf x}_e-A_e,$ for any ${\bf x}_e\in e$,
and the outward normal direction of $e$ is strictly acute by falling into an interval $[0,\theta_0]$ with $\theta_0<\pi/2$.
\item[{\bf A4:}] Assume that each $K\in\mathcal{T}_h$ has a circumscribed simplex $S(K)$ that is shape regular
and has a diameter $h_{S(K)}$ proportional to the diameter of $K$; i.e., $h_{S(K)}\le\gamma_*h_K$ with a constant $\gamma_*$ independent of $K.$
Furthermore, assume that each circumscribed simplex $S(K)$ intersects with only a fixed and small number of such simplexes
for all other elements $K\in\mathcal{T}_h.$
\end{itemize}

Under the above assumptions, the following results have been proved in \cite{WangYe2012}:

\begin{lemma}
(The trace inequality). Assume {\bf{A1}}-{\bf{A3}} hold on a polygonal or polyhedral mesh. Then {\bf{I1}} is true.
\end{lemma}

\begin{lemma}
(The inverse Inequality). Assume {\bf{A1}}-{\bf{A4}} hold on a polygonal or polyhedral mesh and each $V_K$ is the space of
polynomials with degree less than or equal to $n$. Then {\bf{I2}} is true with $C_I$ depending on $n$, but not on $h_K$ or $|K|$.
\end{lemma}

\begin{lemma} \label{lem:L2proj}
Assume {\bf{A1}}-{\bf{A4}} hold on a polygonal or polyhedral mesh and each $V_K$ is the space of
polynomials with degree less than or equal to $n$.
Let $Q_h$ be the $L^2$ projection onto $V_h$. Then for all $0\le s\le n$ and $v\in H^{s+1}(\Omega)$,
$$
\begin{aligned}
\sum_{K\in\mathcal{T}_h}\|v-Q_hv\|_K^2 &\le C_{Q0} h^{2(s+1)} \|v\|_{s+1}^2. \\
\sum_{K\in\mathcal{T}_h}\|\nabla (v-Q_hv)\|_K^2 &\le C_{Q1} h^{2s} \|v\|_{s+1}^2.
\end{aligned}
$$
\end{lemma}

It is not hard to see that {\bf{I3}} follows immediately from Lemma \ref{lem:L2proj}. Indeed, notice that as long as {\bf{I1}} and {\bf{I2}}
are true and $v\in H^r(\Omega)$ with $r>3/2$, we have
$$
\3bar v-Q_h v\3bar \le (1+C_1) \sum_{K\in\mathcal{T}_h}\|\nabla (v-Q_hv)\|_K^2 + \alpha\sum_{e\in\mathcal{E}_h}\frac{1}{h_e}\|[v-Q_h v]\|_e^2,
$$
where $C_1=C_T(1+C_I)^2$.
Next, notice that by {\bf{A2}}, {\bf{I1}} and Lemma \ref{lem:L2proj},
$$
\begin{aligned}
\sum_{e\in\mathcal{E}_h}\frac{1}{h_e}\|[v-Q_h v]\|_e^2 &\le \sum_{K\in\mathcal{T}_h} \left(\sum_{e\in \partial K\cap \mathcal{E}_h} \frac{1}{h_e} \|(v-Q_h v)|_K\|_e^2\right) \\
& \le \sum_{K\in\mathcal{T}_h} \left(\sum_{e\in \partial K\cap \mathcal{E}_h} \frac{1}{\kappa h_K} \|(v-Q_h v)|_K\|_e^2\right) \\
& = \sum_{K\in\mathcal{T}_h} \frac{1}{\kappa h_K} \|(v-Q_h v)|_K\|_{\partial K}^2  \\
& \le \sum_{K\in\mathcal{T}_h} \frac{C_T}{\kappa h_K} \left(h_K^{-1} \|v-Q_h v\|_K^2 + h_K\|\nabla(v-Q_h v)\|_K^2\right) \\
&\le \frac{C_T}{\kappa} (C_{Q0} + C_{Q1}) h^{2s} \|v\|_{s+1}^2.
\end{aligned}
$$
Combine the above, we have

\begin{lemma}
(The aproximability) Assume {\bf{A1}}-{\bf{A4}} hold on a polygonal or polyhedral mesh and each $V_K$ is the space of
polynomials with degree less than or equal to $n$. Then for all $\frac{1}{2}< s\le n$ and $v\in H^{s+1}(\Omega)$,
there exists a constant $C_A$ independent of $h$ such that
$$
\inf_{\chi_h\in V_h} \3bar v-\chi_h \3bar \le C_A h^{s} \|v\|_{s+1}.
$$
Here $s>\frac{1}{2}$ is added so that $\3bar v - \chi_h\3bar$ is well-defined.
\end{lemma}

\section{Numerical Examples}

Finally, we present numerical results that support the theoretical analysis of this paper.
We fix the coefficients $\delta = 1$ and $\alpha = 10$, since the purpose
of the numerical experiments is to examine the accuracy of the interior penalty discontinuous Galerkin method
on arbitrary polygonal meshes, not for different coefficients.
Consider the Poisson's equation on $\Omega = (0,1)\times(0,1)$ with the exact solution
$u = sin(2\pi x)\cos(2\pi y)$. Clearly $u=0$ on $\partial\Omega$.
For simplicity of the notation, we denote
$$
|u-u_h|_{1,h} = \left(\sum_{K\in\mathcal{T}_h} |\nabla(u-u_h)|_K^2\right)^{1/2}.
$$

The first test is performed on a non-conformal triangular-quadrilateral hybrid mesh.
The initial mesh and the mesh after one uniform refinement are given in Figure \ref{fig:mesh1}.
A sequence of uniform refinements are then applied to generate a set of nested meshes.
Notice that the meshes are non-conformal and there are hanging nodes. However, the interior penalty discontinuous Galerkin method
can deal with such meshes without special treatments.
We solve the Poisson equation using the interior penalty discontinuous Galerkin formulation
(\ref{eq:dg}) on these meshes, where the local discrete spaces $V_K$ are taken to be $P_1$ polynomials on each $K\in \mathcal{T}_h$,
no matter whether $K$ is a triangle or quadrilateral.
The $H^1$ semi-norm and the $L^2$ norm of the errors are reported in Table \ref{tab:test1}
and Figure \ref{fig:test1}. These errors are computed using a 5th order Gaussian quadrature on triangles.
For quadrilateral elements, the errors can be conveniently computed by dividing the quadrilateral into two
triangles and then applying the Gaussian quadrature.
Our results show that the $H^1$ semi-norm has an approximate order of $O(h)$, while
the $L^2$ norm has an approximate order of $O(h^2)$, as predicted by the theoretical analysis.

\begin{figure}
\begin{center}
  \caption{Initial and refined mesh for test 1.} \label{fig:mesh1}
  \includegraphics[width=6cm]{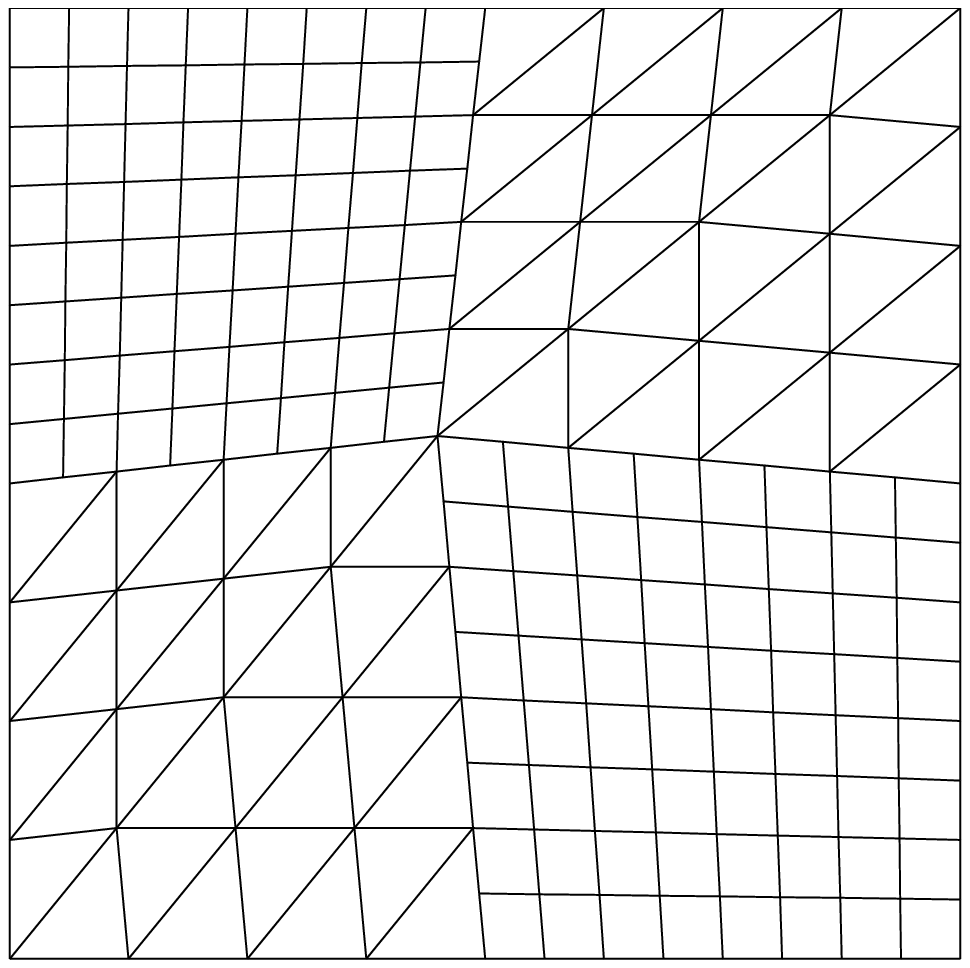}\quad\includegraphics[width=6cm]{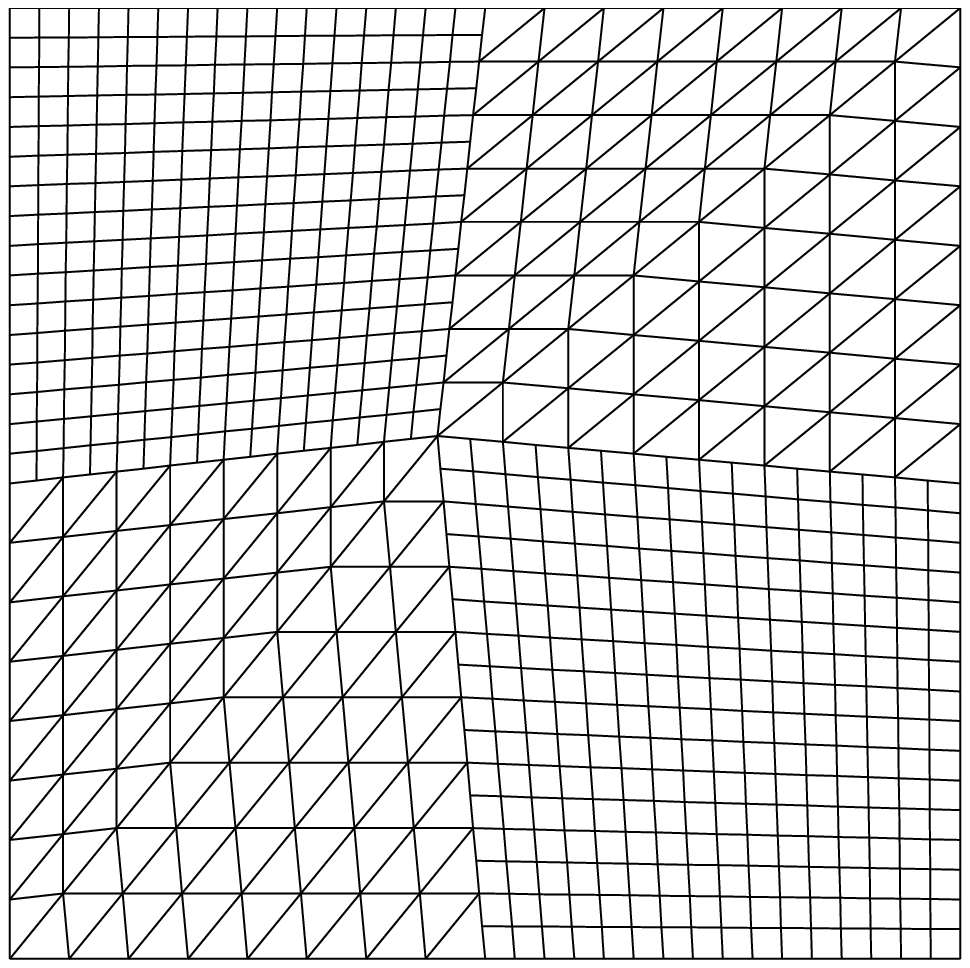}
\end{center}
\end{figure}

\begin{table}
\begin{center}
  \caption{Convergence rates for test 1.} \label{tab:test1}
  \begin{tabular}{|c|c|c|c|c|c|c|}
    \hline
    $h$ & $\frac{1}{16}$ & $\frac{1}{32}$ & $\frac{1}{64}$ & $\frac{1}{128}$ & $\frac{1}{256}$  & $O(h^r)$, $r=$ \\[1mm] \hline
    $|u-u_h|_{1,h}$ & 1.2006 & 0.5904 & 0.2917 & 0.1452 & 0.0725  & 1.0124 \\ \hline
    $\|u-u_h\|$ & 0.0551 & 0.0159 & 0.0042 & 0.0011 & 0.0003  & 1.9270 \\ \hline
  \end{tabular}
\end{center}
\end{table}

\begin{figure}
\begin{center}
  \caption{Convergence rates for test 1.} \label{fig:test1}
  \includegraphics[width=8cm]{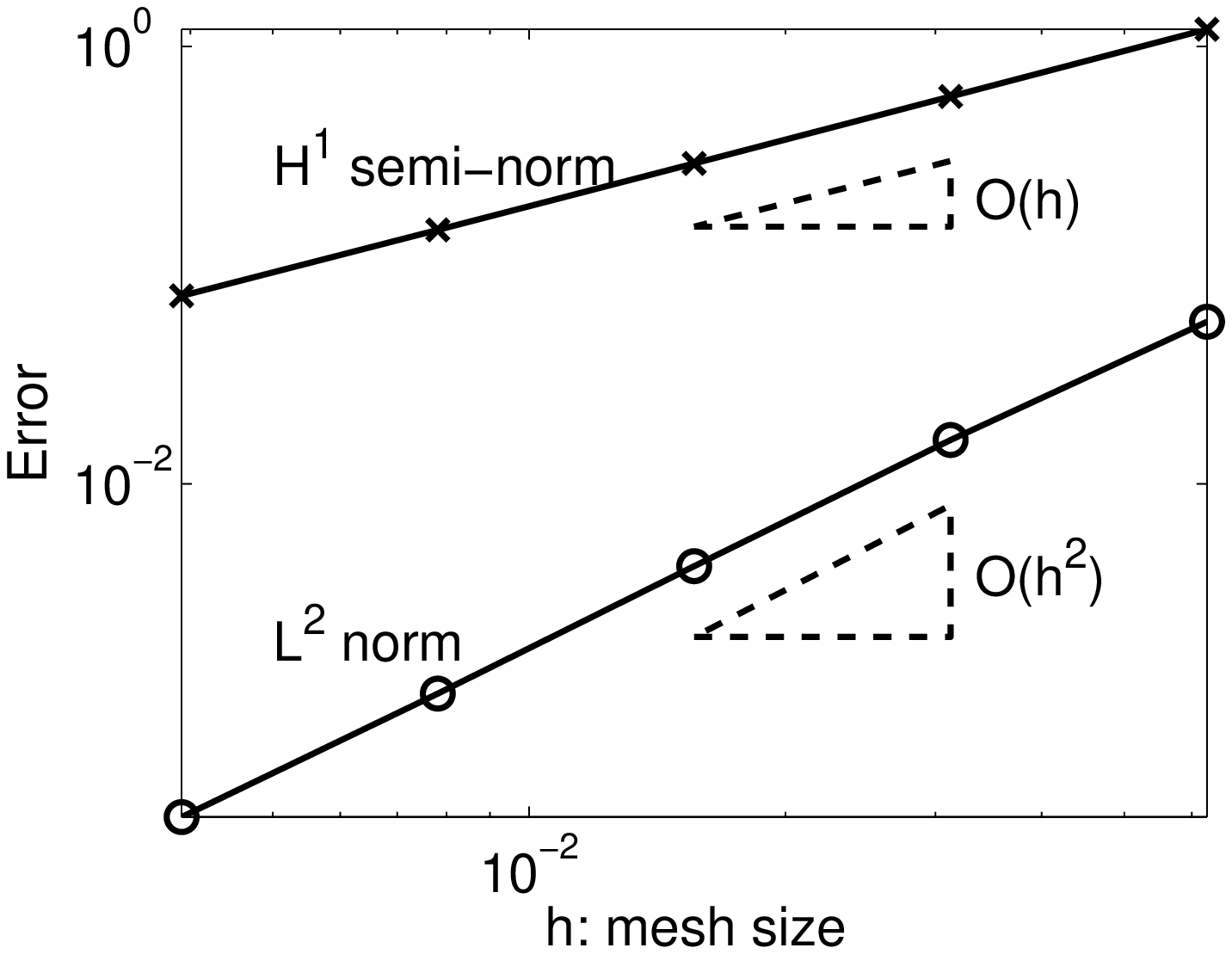}
\end{center}
\end{figure}

In the second test, we consider a hybrid mesh containing mainly hexagons, but with a few quadrilaterals and pentagons.
Indeed, it is derived by taking the dual mesh of a simple triangular mesh.
In Figure \ref{fig:mesh2}, the initial triangular mesh and its dual mesh are shown.
By refining the triangular mesh and computing its dual mesh, we get a sequence of hexagon hybrid meshes.
Again, we solve the interior penalty discontinuous Galerkin formulation
(\ref{eq:dg}) on these hexagon hybrid meshes, with the local discrete spaces $V_K$ of $P_1$ polynomials.
The $H^1$ semi-norm and the $L^2$ norm of the errors are reported in Table \ref{tab:test2}
and Figure \ref{fig:test2}.
Optimal convergence rates are achieved.

\begin{figure}
\begin{center}
  \caption{The original triangular mesh and its dual mesh used in test 2.} \label{fig:mesh2}
  \includegraphics[width=6cm]{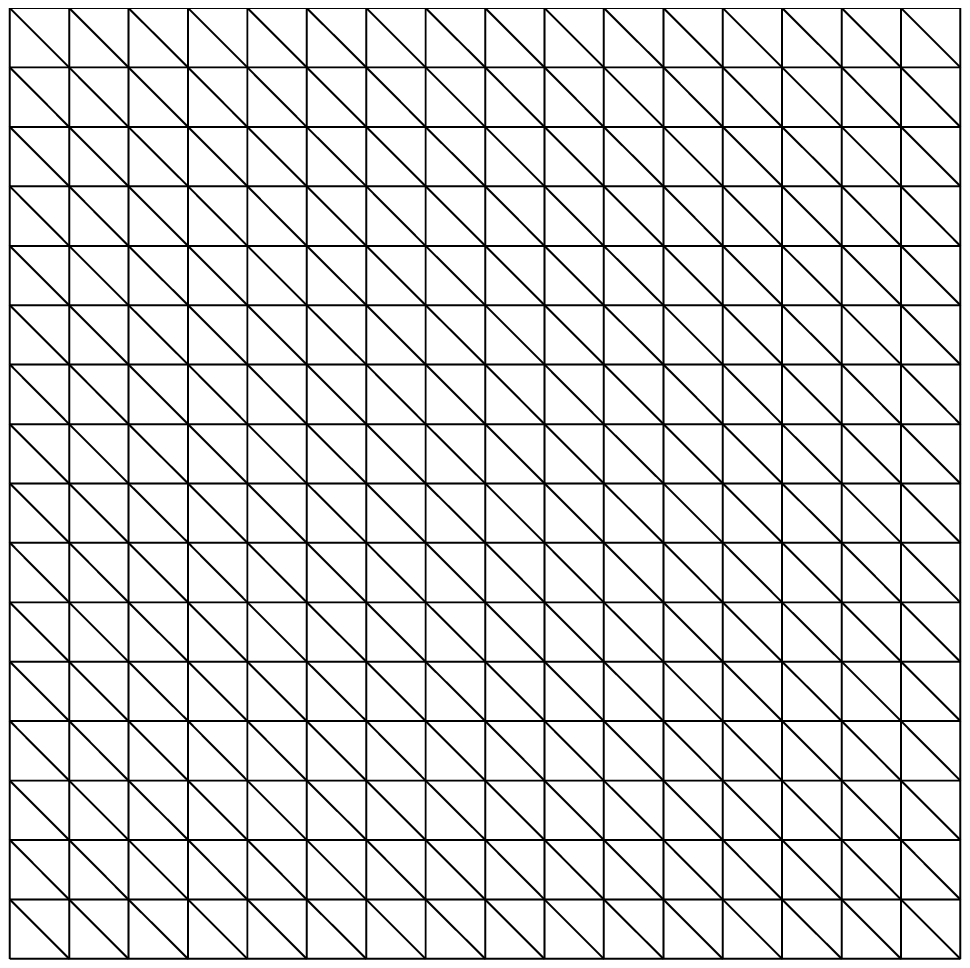}\quad\includegraphics[width=6cm]{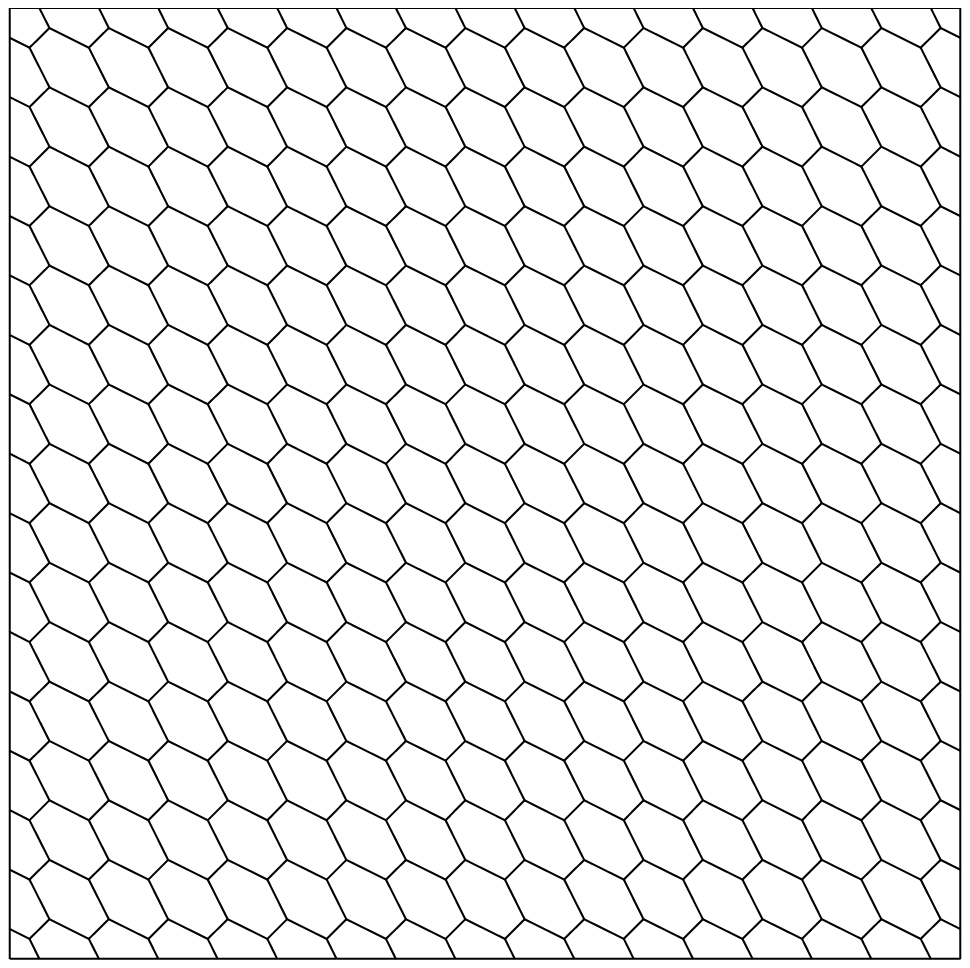}
\end{center}
\end{figure}

\begin{table}
\begin{center}
  \caption{Convergence rates for test 2.} \label{tab:test2}
  \begin{tabular}{|c|c|c|c|c|c|c|}
    \hline
    $h$ & $\frac{1}{16}$ & $\frac{1}{32}$ & $\frac{1}{64}$ & $\frac{1}{128}$ & $\frac{1}{256}$  & $O(h^r)$, $r=$ \\[1mm] \hline
    $|u-u_h|_{1,h}$ & 0.8139 & 0.3868 & 0.1894 & 0.0941 & 0.0470 & 1.0270\\ \hline
    $\|u-u_h\|$ & 0.0461 & 0.0129 & 0.0034 & 0.0009 & 0.0002 & 1.9393\\ \hline
  \end{tabular}
\end{center}
\end{table}

\begin{figure}
\begin{center}
  \caption{Convergence rates for test 2.} \label{fig:test2}
  \includegraphics[width=8cm]{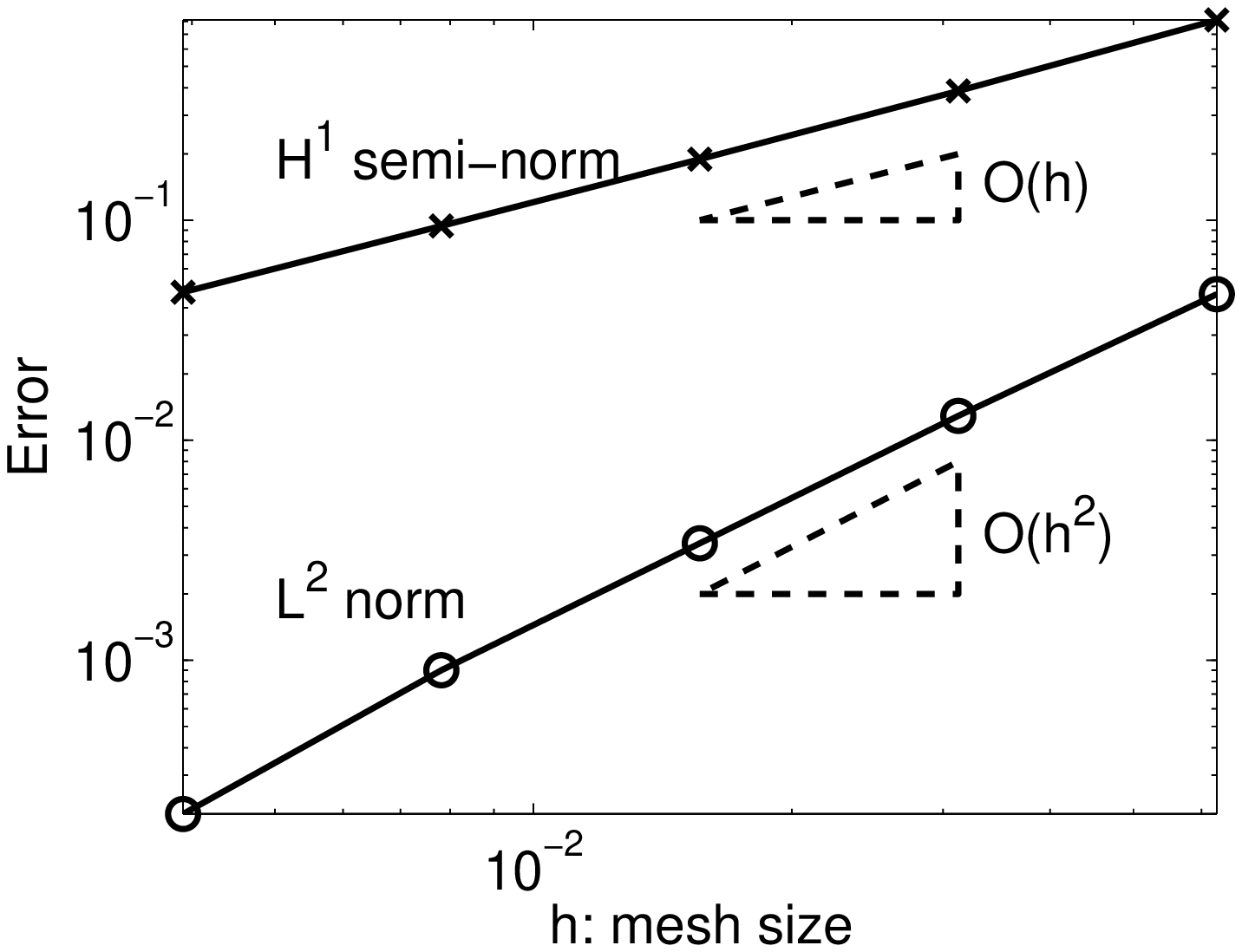}
\end{center}
\end{figure}


\end{document}